\documentclass[reqno]{amsart}

\usepackage{amsmath,amssymb,amsthm}

\usepackage{tikz}

\usepackage{caption}
\usepackage{graphicx}
\usepackage{graphics}

\usepackage{hyperref}

\newtheorem{thm}{Theorem}

\newtheorem*{corollary}{Corollary}

\usepackage{url}
\usepackage{xcolor}

\begin{document}

\title[]{A random line intersects {\Large $\mathbb{S}^2$} in two probabilistically independent locations}

\author[]{Dmitriy Bilyk \and Alan Chang \and Otte Hein\"avaara \and \\ Ryan W. Matzke \and Stefan Steinerberger
}
\address{School of Mathematics, University of Minnesota, Minneapolis, MN 55455, USA}
\email{dbilyk@umn.edu}

\address{Department of Mathematics, Washington University in St. Louis, St. Louis, MO 63105, USA}
\email{alanchang@math.wustl.edu}

\address{Department of Mathematics, Princeton University, Princeton, NJ 08544, USA}
\email{oeh@math.princeton.edu}

\address{Department of Mathematics, Vanderbilt University, Nashville, TN 37235, USA}
\email{ryan.w.matzke@vanderbilt.edu}

\address{Department of Mathematics, University of Washington, Seattle, WA 98195, USA}
 \email{steinerb@uw.edu}

\thanks{D. B. is supported by the NSF grant DMS-2054606. R.W.M. is supported by NSF Postdoctoral Fellowship Grant 2202877. S.S. is supported by the NSF (DMS-2123224).}

\begin{abstract} We consider random lines in $\mathbb{R}^3$ (random with respect to the kinematic measure) and how they intersect $\mathbb{S}^2$. It is known that the entry point and the exit point behave like \textit{independent} uniformly distributed random variables. We give a new proof using bilinear integral geometry and use this approach to show that this property is extremely rare: if $K \subset \mathbb{R}^n$ is a bounded, convex domain with smooth boundary with this property (i.e., the intersection points with a random line are independent), then $n=3$ and $K$ is a ball.
\end{abstract}
\maketitle

\vspace{10pt}

\section{Introduction and results}
The purpose of this paper is to investigate a curious phenomenon on $\mathbb{S}^2$: ``random'' lines intersecting $\mathbb S^2$ do so in two (probabilistically) \textit{independent} points, and this property  is unique to  $\mathbb S^2$ among all boundaries of smooth convex sets in any dimension. Here and henceforth, ``random'' line will always refer to directed lines $\ell$ in  $\mathbb{R}^3$
that are chosen with respect to the unique measure that is invariant under translations and rotations, also known
as the ``kinematic'' measure, which we normalize so that the set of lines
that intersect $\mathbb{S}^2$ has measure one, i.e., it is a probability measure on this set. 
The kinematic measure corresponds to picking the direction of $\ell$ uniformly over $\mathbb{S}^2$ and then choosing  the intersection point of $\ell$ with the two-dimensional subspace orthogonal to this direction according to the Lebesgue measure on that plane. 
Our starting point is the following Theorem.
\begin{thm}[Akopyan, Edelsbrunner, Nikitenko \cite{ak}]\label{t1}
The probability distribution on lines intersecting $\mathbb{S}^2$, defined by choosing two points uniformly and independently on $\mathbb{S}^2$, coincides with the Crofton measure.
\end{thm}

Phrased differently, if we know that a random line enters $\mathbb{S}^2$ in a certain region $A$, we cannot deduce any information about its exit point (which is also uniformly distributed over $\mathbb{S}^2$). The statement, once suspected to be true, is not difficult to verify and this can be done in a number of different ways. Akopyan, Edelsbrunner and Nikitenko \cite[Lemma 8]{ak} give a very slick proof (summarized below). There is also a fairly direct calculus proof that computes all the surface elements and induced projections, essentially a very simple special case of \cite[Lemma 1]{bushling}. Given the fundamental nature of the statement and the relative ease of the proofs, it may have appeared, explicitly or implicitly, many more times in the literature (see, for example, Burdzy \& Rizzolo \cite[Appendix A]{burdzy}). We will give yet another proof based on recent ideas in bilinear integral geometry. Our main contribution is to show that this property is so rare that it \textit{characterizes} $\mathbb{S}^2$ among domains with twice differentiable boundary.

 \begin{thm}\label{t2}
Let $K$ be a bounded, convex domain with $C^2$ boundary $\partial K$ in $\mathbb{R}^{n}$
and suppose that the measure induced by picking independent points on $\partial K$ independently and uniformly coincides with the Crofton measure. Then $n=3$ and $\partial K = \mathbb{S}^2$ (up to translation and dilation).
\end{thm}

It seems exceedingly likely that the requirement that $\partial K$ is $C^2$ can be dropped but this would require some new ideas. The remainder of the paper is structured as follows.
We first recall the elementary proof of Theorem \ref{t1} given by Akopyan, Edelsbrunner and Nikitenko \cite{ak} in Section \ref{s.elem}. Then, in Section \ref{s.crof}, we give an alternative proof of this theorem, which further explains the nature of the phenomenon and is based on tools of integral geometry, such as the Crofton formula and its extensions. Finally, in Section \ref{s.t2} we prove Theorem \ref{t2}. In what follows, $\mathcal{H}^n$ denotes the $n$-dimensional Hausdorff measure and  $\sigma$ is the uniform surface measure on $\mathbb S^{n-1}$ (i.e., a restriction of $\mathcal{H}^{n-1}$ to $\mathbb S^{n-1}$), normalized so that $\sigma (\mathbb S^{n-1})=1$. 
The normalized kinematic measure is denoted by $\mu$ or, when appropriate, simply using the probabilistic notation $\mathbb P$.

\section{Quick Proof of Theorem \ref{t1}}\label{s.elem}
This section quickly provides the argument given by Akopyan, Edelsbrunner and Nikitenko \cite[Lemma 8]{ak}.
 It uses an ancient result of Archimedes \cite{A} that the area of a spherical cap on $\mathbb{S}^2 \subset \mathbb{R}^3$ is
\begin{equation}\label{e.arch}
 \mathcal{H}^2\left( \left\{ x \in \mathbb{S}^2: x_3 \geq t \right\} \right) = 2\pi(1-t).
\end{equation}
This fact is sometimes referred to as the Archimedes hat-box theorem and implies that in $\mathbb R^3$ the horizontal slices of the  sphere and the circumscribed cylinder  have equal areas. An illustration of this fact has allegedly been engraved on Archimedes's tombstone discovered by Cicero \cite{W} (see \cite{nimer} for a proof of Archimedes's theorem as well as another recent application to geometric measure theory).

\begin{proof}[Proof of Theorem \ref{t1} from \cite{ak}]
 Let $\ell$ be a directed line and let $X$ and $Y$ denote the first and the second time $\ell$ intersects $\mathbb{S}^{2}$. We claim that $X$ and $Y$ are independent random variables, uniform on $\mathbb{S}^{2}$.
The claim on uniformity is clear in view of rotational invariance of the kinematic measure and the sphere. 
Also, after conditioning on $X$, the distribution of $Y$ is invariant under all rotations of $\mathbb{S}^{2}$ fixing $X$. It is therefore enough to check that distribution of $\|X - Y\|$ is the same as that of $\|X - X_{0}\|$ for any fixed $X_{0} \in \mathbb{S}^{2}$.
Denote by $Z$ the midpoint of $X$ and $Y$. We note that the $Z$ is the intersection of $\ell$ with the plane orthogonal to $\ell$ and passing through $0$. By the definition of the kinematic measure, after conditioning on the plane, this point is uniform on the projected disc. Hence we have
\begin{align*}
    \mathbb{P}(\|X - Y\| \leq d) = \mathbb{P}\left(\|Z\| \geq \sqrt{1 - \frac{d^2}{4}}\right) = 1 - \left(\sqrt{1 - \frac{d^2}{4}}\right)^{2} = \frac{d^2}{4}.
\end{align*}

We will now calculate $\mathbb{P}(\|X - X_{0}\| \leq d)$ for a fixed $X_0 \in \mathbb S^2$. 
Observe that the inequality $\|X - X_{0}\| \leq d$ is equivalent to $\langle X,X_0 \rangle \ge 1- d^2/2$,
which   allows us to use \eqref{e.arch} and to calculate
\begin{align*}
    \mathbb{P}(\|X - X_{0}\| \leq d) = \mathbb{P}\left(\langle X,X_0 \rangle \ge 1- \frac{d^2}{2} \right) = \frac12 \left( 1 - \left( 1- \frac{d^2}{2} \right) \right) =  \frac{d^2}{4}.
\end{align*}
\end{proof}

\section{Integral-geometric proof of Theorem \ref{t1}}\label{s.crof}
In this section we present another proof of Theorem \ref{t1} based on tools of integral geometry, which has the advantage of building additional intuition. We start with a simple immediate Corollary to Theorem \ref{t1} for which we give an alternative proof. The Corollary can be used to produce yet another independent proof of  Theorem \ref{t1}. This is the original route through which Theorem \ref{t2} was first discovered. We
prove the Corollary in \S \ref{ss.proof}. Another proof of Theorem \ref{t1} is given in \S \ref{ss.2proof}.

\begin{corollary}
Let $A \subset \mathbb{S}^2$ be measurable and let $\ell$ be a line chosen uniformly at random with respect 
to the kinematic measure on the set of all lines intersecting $\mathbb{S}^2$. Then, using $\sigma$ to denote
the normalized measure on $\mathbb{S}^2$,
\begin{align}\label{e.A2}
\mathbb{P}\left(\ell~\emph{intersects}~A~\emph{in two points}\right) &=  \sigma(A)^2.
\end{align}
\end{corollary}
By applying the corollary to $A^c$ we deduce that
\begin{equation}\label{e.Ac2}
    \mathbb{P}\left(\ell~\mbox{does not intersect}~A\right) =  (1-\sigma(A))^2
    \end{equation}
from which we deduce that
$$\mathbb{P}\left(\ell~\mbox{intersects}~A~\mbox{in one point}\right) =  2\sigma(A)(1-\sigma(A)) =  2\sigma(A)\sigma(A^c).$$

\subsection{Quadratic Crofton.} \label{ss.qcrof}
We start by developing some background. 
We briefly discuss Crofton's formula and its extension to a certain quadratic  integral identity, which will be used  to prove the Corollary.  This  identity will also allow for a discussion of $\mathbb{S}^{n-1}$ when $n \neq 3$. We fix a two-dimensional, rectifiable set $A \subset \mathbb{R}^3$ and consider ``random'' lines with respect to the kinematic measure $\mu$. For any given line $\ell$, we denote the number of intersections of the set $A$ with the line $\ell$ by $n_{\ell}(A)$. The Cauchy--Crofton formula in $\mathbb{R}^3$ says that the expected number of intersections only depends on the surface area of the set and 
\begin{equation}\label{e.crof}
    \mathcal{H}^2(A)= 2 \pi \int n_{\ell}(A) d\mu(\ell),
\end{equation}
where $\mathcal{H}^2$ denotes the two-dimensional Hausdorff measure. For more on the Cauchy--Crofton formula, see e.g., \cite[(13.14)]{santa2}, \cite[Theorem 9.7]{fed1}, or \cite[Theorem 3.2.26]{fed2}. The normalization $2\pi$ might look slightly unusual -- this is the result of our normalization of the kinematic measure $\mu$. The value of the constant is easily obtained from the case $A= \mathbb S^2$, since $\mathcal{H}^2 (\mathbb S^2) = 4\pi$ and almost all lines intersecting $\mathbb S^2$  have exactly two intersection points (the $\mu-$measure of this set of lines is one). \\
  The second ingredient is a recently discovered identity (see \cite{chang, stein, stein2}), valid for $(n-1)$-dimensional surfaces $A \subset \mathbb{R}^n$ of regularity $C^2$: for two universal constants $c_n$, $c_n^* > 0$ that only depend on the dimension we have that
 \begin{align}\label{eq:Square Crofton}
c_n \int n_{\ell}(A)^2 d\mu(\ell) & - \mathcal{H}^{n-1}(A) = \\ \nonumber & =  c_n^* \int_{A} \int_{A}\frac{\left|\left\langle n(x), y - x \right\rangle  \left\langle  y - x, n(y) \right\rangle   \right| }{\|x - y\|^{n+1}}~d \mathcal{H}^{n-1} (x) d\mathcal{H}^{n-1} (y),
\end{align}
where $n(x)$ denotes the normal vector to $A$ at $x$. 
The assumption in \cite{stein, stein2} was that $A$ is at least $C^2$. This was recently extended by Bushling \cite{bushling} to the case of $(n-1)-$rectifiable sets. 
By considering the case where $A$ is a subset of a hyperplane (in which case the right-hand side is 0), we note that $n_{\ell}(A) \in \left\{0,1\right\}$ in which case the identity reduces to the Cauchy--Crofton formula and we deduce that $c_3 = 2\pi$. 

\subsection{Proof of the Corollary}\label{ss.proof}
\begin{proof}
Let us fix a measurable set $A \subset \mathbb{S}^2$. 
It suffices to determine the likelihood that a random line $\ell$ intersects $A$ twice. Each line intersects $A$ either 0, 1 or 2 times. The Cauchy--Crofton formula tells us that the expected number of intersections is linearly proportional to the surface area:  
using \eqref{e.crof}, we obtain 
\begin{align} \label{eq:one}
\mathbb{E}(n_{\ell}(A)) = \mathbb{P}(n_{\ell}(A) = 1) + 2 \cdot  \mathbb{P}(n_{\ell}(A) = 2) = \frac{\mathcal{H}^2(A)}{2\pi}.
 \end{align}
Our second ingredient will be the quadratic Crofton formula: it is applicable because $A \subset \mathbb{S}^2$ is necessarily 2-rectifiable (because it is a subset of $\mathbb{S}^2$ which is $2-$rectifiable) and thus  \cite[Lemma 1]{bushling} applies. We note that for points $x \in \mathbb{S}^2$, the normal vector is particularly simple:  $n(x) = x$. Thus the integrand simplifies to
\begin{align*}
\frac{\left|\left\langle n(x), y - x \right\rangle  \left\langle  y - x, n(y) \right\rangle   \right| }{\|x - y\|^{4}} &= \frac{\left(1 - \left\langle x,y \right\rangle\right)^2 }{(\|x - y\|^{2})^2} \\
&=  \frac{\left(1 - \left\langle x,y \right\rangle\right)^2 }{(2 - 2 \left\langle x, y\right\rangle)^2} = \frac{1}{4}.
\end{align*}
Therefore, for some universal constant $\beta >0$,
\begin{align} \label{eq:two}
 \mathbb{P}(n_{\ell}(A) = 1) + 4 \cdot  \mathbb{P}(n_{\ell}(A) = 2) =  \frac{\mathcal{H}^2(A)}{2\pi} + \beta \cdot \big(\mathcal{H}^2(A) \big)^2.
\end{align}
Using \eqref{eq:one} and \eqref{eq:two}, we deduce
$$  \mathbb{P}(n_{\ell}(A) = 2)  = \frac{\beta}{2} \cdot \big(\mathcal{H}^2(A)\big)^2.$$
It remains to compute the constant $\beta$. Considering the case where $A = \mathbb{S}^2$, we have $ \mathbb{P}(n_{\ell}(A) = 2)  = 1$ and $\mathcal{H}^2(A) = 4\pi$, so $\beta = 1/8 \pi^2$. Recalling that the normalized surface measure is given by $\sigma(A) = \mathcal{H}^2(A)/4\pi$, we see that
$$  \mathbb{P}(n_{\ell}(A) = 2)  = \frac{ \big( \mathcal{H}^2(A)\big)^2}{16 \pi^2} = \left(  \frac{ \mathcal{H}^2(A)}{4\pi}\right)^2 = \sigma(A)^2.$$
\end{proof}

\subsection{Second proof of Theorem \ref{t1}}\label{ss.2proof}
\begin{proof}[Proof]
Let $A, B \subset \mathbb{S}^2$ be two disjoint sets 
and consider their union $C = A \cup B$. The Corollary implies that
$$ \mathbb{P}(\ell \text{ intersects } C \text{ at two points}) = \sigma(C)^2.$$
However, we also have that
$$ \sigma(C)^2 = \sigma(A)^2 + \sigma(B)^2 + 2\sigma(A)\sigma(B).$$
Since $\sigma(A)^2$ and $\sigma(B)^2$ are the probabilities that $\ell$ intersects $A$ or $B$ twice, respectively,  we deduce that $2 \sigma(A) \sigma(B)$ is the probability that $\ell$ intersects $A$ and $B$ each exactly once.  Thus for oriented lines $$ \mathbb P (X (\ell) \in A, Y(\ell ) \in B) = \sigma (A) \sigma (B). $$   
If the sets $A$, $B$ are not disjoint, one can finish the proof by using the fact that the two points $X(\ell), Y(\ell)$ inducing the line $\ell$ satisfy $\mathbb P (X (\ell),  Y(\ell ) \in  A \cap B) = \sigma (A \cap B)^2$ and using  the inclusion-exclusion principle. Alternatively, one could just directly deduce that the distribution of $(X (\ell),  Y(\ell ) )$ is absolutely continuous and thus is  equal to $\sigma \times \sigma$.
\end{proof}

\subsection{Higher dimensions}
A nice aspect of the proof is that it can be easily adapted to higher dimensions 
If $A \subset \mathbb{S}^{n-1} \subset \mathbb{R}^n$, we deduce that, for dimensional
constants $\beta_n, \gamma_n$, we have
$$  \mathbb{P}(n_{\ell}(A) = 1) + 2 \cdot  \mathbb{P}(n_{\ell}(A) = 2) =  \beta_n \mathcal{H}^2(A)$$
as well as
$$  \mathbb{P}(n_{\ell}(A) = 1) + 4 \cdot  \mathbb{P}(n_{\ell}(A) = 2) =  \beta_n \mathcal{H}^2(A) + \gamma_n \int_A \int_A \frac{1}{\|x-y\|^{n-3}} d\sigma(x) d\sigma(y)$$
from which we deduce
$$  \mathbb{P}(n_{\ell}(A) = 2) = \frac{\gamma_n}{2} \int_A \int_A \frac{1}{\|x-y\|^{n-3}} d\sigma(x) d\sigma(y)$$
which, unless $n=3$, depends on the set $A$ and not just its measure. By setting $A = \mathbb{S}^{n-1}$, we deduce
$$  \mathbb{P}(n_{\ell}(A) = 2) = \left(  \int_{\mathbb{S}^{n-1}} \int_{\mathbb{S}^{n-1}} \frac{1}{\|x-y\|^{n-3}} d\sigma d\sigma \right)^{-1} \int_A \int_A \frac{1}{\|x-y\|^{n-3}} d\sigma d\sigma.$$
In particular, we see that the probability of a random line intersecting a set $A$ in two points very much depends on both the set and the dimension unless $n=3$. 
The argument of \S \ref{ss.2proof} generalizes and implies that on $\mathbb{S}^{n-1}$, for disjoint $A$ and $B$,
$$ \mathbb{P}(A \cap \ell \neq \emptyset \text{ and } B \cap \ell \neq \emptyset) = c_n \int_{A} \int_{B} \frac{1}{\|x-y\|^{n-3}} d\sigma(x) d \sigma(y),$$
where $c_n$ is a normalization factor that only depends on the dimension. In particular, when $n = 2$, a line passing through $A$ is more likely to also pass through $B$ if they are farther apart. For $n \geq 4$, a line passing through $A$ is more likely to pass through $B$ the closer $B$ is to $A$. It is also a straightforward exercise  to compute that $\mathbb E \| X- Y\| \sim n^{-1/2}$, i.e., in high dimensions the entry and exit points are typically very close. At the same time, two \textit{independent} uniform random points would be nearly orthogonal due to the   \textit{concentration of measure} phenomenon (see, e.g., \cite{ledoux}, \cite[Chapter 14]{matousek}). 
It is even easier to compute that $\mathbb E \langle X,Y \rangle = \frac{n-3}{n+1}$, which is zero only if $n=3$ (and approaches one for large $n$), while for two \textit{independent} uniform random points on the sphere this expectation is obviously zero. 

\section{Proof of Theorem \ref{t2}}\label{s.t2}
\begin{proof}
Let $K$ be a bounded, convex set in $\mathbb{R}^{n}$. We also assume that $K$ has nontrivial $n$-dimensional volume $\mathcal{H}^n(K) > 0$ and that its boundary $\partial K$ is $C^2$.  We consider random lines $\ell$ from the kinematic measure conditioned to hit $K$ and use $(X, Y)$ be the ordered pair of random intersection points of $\ell$ and the boundary $\partial K$. By assumption, if $A$ and $B$ are small disjoint domains on $\partial K$ then, because the random variables are independent,
\begin{align*}
    \mathbb{P}(X \in A \text{ and } Y \in B) = \mathbb{P}(X \in A) \cdot \mathbb{P}(Y \in B).
\end{align*}
We will use this fact for two very small domains (say, geodesic balls on the surface) where we may think of one of these domains as fixed and the second one approaching the first one. The computation in \cite{stein} shows that for disjoint sets, $A \cap B = \emptyset$,
$$     \mathbb{P}(X \in A \text{ and } Y \in B)  =    c_n\int_{A \times B} \frac{|\langle n(x), y - x \rangle \langle n(y), x - y \rangle|}{\|y - x\|^{n + 1}} d \mathcal{H}^{n - 1}(x) d \mathcal{H}^{n - 1}(y).$$ Since, according to Crofton's formula,  $\mathbb P ( X\in A) $ is proportional to $\mathcal{H}^{n - 1} (A)$,
taking the limit as $A$ and $B$ both shrink down to distinct points  forces the function $F:\partial K \times \partial K \rightarrow \mathbb{R}$ given by
$$     F(x, y) = \frac{|\langle n(x), y - x \rangle \langle n(y), x - y \rangle|}{\|y - x\|^{n + 1}}
$$
to be constant. We will now think of $x \in \partial K$ as fixed and $y \in \partial K$ as a point very close to $x$ and deduce additional information from the fact that $F$ is constant. Since the problem is invariant under rotation and translation, we may assume that $\partial K \ni x = \mathbf{0} \in \mathbb{R}^n$ is rotated in such a way that the tangent plane is given by $\left\{z \in \mathbb{R}^n: z_n = 0\right\}$ and that 
$ K \subset \left\{z \in \mathbb{R}^n: z_n \leq 0\right\}$ is a subset of the lower half-space.
Since $\partial K$ is $C^2$, the tangent plane is unique. We can furthermore write $\partial K$ in a neighborhood of $x = \mathbf{0} \in \mathbb{R}^n$  as the graph of a function $\phi:\mathbb{R}^{n-1} \rightarrow \mathbb{R}$ where
$$ \phi(z_1, \dots, z_{n-1}) = \frac{1}{2} \left\langle z, Qz \right\rangle + o(\|z\|^2),$$
where $Q \in \mathbb{R}^{(n-1) \times (n-1)}$ is a symmetric and negative semi-definite matrix. This means that, up to an error term we can write $\partial K$ as $(z, \phi(z))$ for $z \in \mathbb{R}^{n-1}$ very close to the origin. The normal direction to $\partial K$ can also be computed: we note, for small $\delta \in \mathbb{R}^{n-1}$ we have, because $Q$ is symmetric,
\begin{align*}
 \phi(z+ \delta) &= \frac{1}{2} \left\langle z + \delta, Q (z+\delta) \right\rangle + o(\|z\|^2) \\
 &= \phi(z) + \left\langle \delta, Qz \right\rangle + \mathcal{O}(\|\delta\|^2) + o(\|z\|^2).
 \end{align*}
 This means the local tangent plane in $(z, \phi(z))$ is given by 
 $$ (z + \delta, \phi(z+\delta)) = (z,\phi(z)) + (\delta,\left\langle \delta, Qz \right\rangle) +  \mbox{lower order terms.}$$
We note that the linear term is always orthogonal to the vector $\left( -Qz, 1 \right)$ since
$$ \left\langle  (\delta,\left\langle \delta, Qz \right\rangle), \left( -Qz, 1 \right) \right\rangle = - \left\langle \delta, Qz \right\rangle + \left\langle \delta, Qz \right\rangle = 0.$$
For $z \in \mathbb{R}^{n-1}$ close to 0, the normal vector oriented pointing upwards is given by
$$ n(z) = \frac{\left( -Qz, 1 \right)}{\| (-Qz, 1)\|} +o(\|z\|).$$
We remark that the direction of the normal vector is not important in what follows, it could have equally been chosen to point downwards. We also note, again for $z$ close to $\mathbf{0} \in \mathbb{R}^{n-1}$
$$ \| (-Qz, 1)\| = \sqrt{1 + \|Qz\|^2} = 1 + o(\|z\|)$$
allowing us to write, for $z$ close to $\mathbf{0} \in \mathbb{R}^{n-1}$,
$$ n(z) = \left( -Qz, 1 \right) + o(\|z\|).$$
We will now simplify the terms in the expression
$$     F(x, y) = \frac{|\langle n(x), y - x \rangle \langle n(y), x - y \rangle|}{\|y - x\|^{n + 1}}
$$
in the case where $x = \mathbf{0}$ and $y$ is very close to $x$. We will also write $y_{n-1}$ to denote the first $(n-1)-$coordinates of $y$ and note that then $y = (y_{n-1}, \phi(y_{n-1}))$.
We start by noting that $n(x) = (0,0,0,\dots,0,1)$ and thus 
$$\left| \langle n(x), y - x \rangle \right| = \left| \frac{1}{2}\left\langle y_{n-1}, Qy_{n-1} \right\rangle \right| + o(\|y_{n-1}\|^2).$$
Likewise
$$ \left| \langle n(y), x - y \rangle \right| =  \left| \langle n(y), - y \rangle \right| =  \left| \langle n(y), y \rangle \right|$$
and we note that, up to lower order error terms, $n(y) = (-Qy_{n-1}, 1)$ while the point is given by $y = (y_{n-1}, \phi(y_{n-1}))$ and thus
\begin{align*}
 \langle n(y), y \rangle &= \left\langle -Q y_{n-1}, y_{n-1} \right\rangle + \phi(y_{n-1})  + o(\|y_{n-1}\|^2) \\
 &= -\frac{1}{2} \left\langle y_n, Q y_n \right\rangle  + o(\|y_{n-1}\|^2)
 \end{align*}
and thus
$$  \left| \langle n(y), y \rangle \right| = \frac{1}{2} \left| \left\langle y_{n-1}, Qy_{n-1} \right\rangle\right| +  o(\|y_{n-1}\|^2).$$
Finally, we argue that, up to lower order error terms
$$ \|y - x\|^{n + 1} =\|y_{n-1}\|^{n+1} + \mathcal{O}(\|y_{n-1}\|^{n+3}).$$
Thus, as $y \rightarrow 0$, we have
$$ \frac{|\langle n(x), y - x \rangle \langle n(y), x - y \rangle|}{\|y - x\|^{n + 1}}  = \frac{1}{4} \frac{  \left| \langle Q y_{n-1}, y_{n-1} \rangle \right|^2}{\|y_{n-1}\|^{n+1}} + o(\|y_{n-1}\|^2).$$
In order for this to be a constant, we require $n=3$ simply by scaling. Knowing this, the situation immediately simplifies since we know now that we are dealing with a two-dimensional surface in $\mathbb{R}^3$. Since $Q \in \mathbb{R}^{2 \times 2}$ is a symmetric matrix, the spectral theorem implies that we can diagonalize the matrix and assume that the eigenvalues are given by $\lambda_1, \lambda_2 \leq 0$ with corresponding eigenvectors $v_1, v_2 \in \mathbb{R}^2$. Plugging in $y_{n-1} = \varepsilon v_i$, we get
 $$   \frac{\left| \langle Q y, y \rangle \right|^2}{\|y\|^{4}} =  \lambda_i^4.$$
This forces the eigenvalues to be the same, $\lambda_1 = \lambda_2$, and thus $Q = - c \cdot \mbox{Id}_{2 \times 2}$ for some constant $c \geq 0$. This means that the principal curvatures are the same which makes the point an \textit{umbilical point}. It is known that this implies that the surface is a subset of a sphere or the subset of a plane (see, for example, \cite{umb}) and this concludes our proof. Another way of concluding the argument is as follows: since the principal curvatures are the same everywhere, the mean curvature is constant. Aleksandrov \cite{ale} proved that if a closed, connected $C^2$ surface has constant mean curvature, then it is a sphere.  
\end{proof}

\textbf{Acknowledgment.} We are grateful for discussions with 
Arseniy Akopyan and Krzysztof Burdzy.

\end{document}